\documentclass[14pt]{amsart}
  \usepackage{amssymb}
 \usepackage{amstext}
 \usepackage{amsmath}
 \usepackage{amscd}
 \usepackage{latexsym}
 \usepackage{amsfonts}
 \usepackage{comment}
 \usepackage{color}
  \usepackage{enumerate}
 \usepackage[all]{xy}
 \usepackage{graphicx}

\theoremstyle{plain}
\newtheorem{Theorem}{Theorem}[section]

\newtheorem{theorem}[Theorem]{Theorem}

\newtheorem{lemma}[Theorem]{Lemma}

\newtheorem*{thm*}{Theorem}
\newtheorem*{cor*}{Corollary}

\newtheorem*{claim*}{Claim}

\theoremstyle{definition}

\newtheorem{definition}[Theorem]{Definition}

\newtheorem{example}[Theorem]{Example}

\newtheorem{Setting}[Theorem]{Setting}
\newtheorem{notation}[Theorem]{Notation}

\newtheorem{fact}[Theorem]{Fact}

\theoremstyle{remark}

\newcommand{\Proof}{\begin{proof}}
\newcommand{\Qed}{\end{proof}}

\numberwithin{equation}{Theorem}

 \def\e{\mathrm{e}}

\newcommand{\calD}{\mathcal{D}}

\newcommand{\calF}{\mathcal{F}}

\newcommand{\calN}{\mathcal{N}}

\newcommand{\calP}{\mathcal{P}}

\newcommand{\calS}{\mathcal{S}}

\newcommand{\fkm}{\mathfrak{m}}

\newcommand{\fkp}{\mathfrak{p}}
\newcommand{\fkq}{\mathfrak{q}}
\newcommand{\Q}{\mathfrak{q}}

\newcommand{\q}{\mathfrak{q}}
\newcommand{\m}{\mathfrak{m}}

\def\Ass{\operatorname{Ass}}
\def\Assh{\operatorname{Assh}}

\def\adeg{\operatorname{adeg}}

\def\reg{\operatorname{reg}}

\begin{document}

\setlength{\baselineskip}{17pt}
\title{On Hilbert coefficients and sequentially generalized Cohen-Macaulay modules }
\author{Nguyen Tu Cuong}

\address{Institute of Mathematics, Vietnam Academy of Science and Technology, 18 Hoang Quoc Viet Road, 10307 Hanoi, Vietnam}
\email{ntcuong@math.ac.vn}

\author{Nguyen Tuan Long}
\address{National Economics University, 207 Giai Phong Road, Hanoi, Vietnam }
\email{ntlong81@gmail.com}

 	\author[H. L. Truong]{Hoang Le Truong}
 	\address{Institute of Mathematics, VAST, 18 Hoang Quoc Viet Road, 10307
 		Hanoi, Vietnam}
 	\address{Thang Long Institute of Mathematics and Applied Sciences, Hanoi, Vietnam}
 	\email{hltruong@math.ac.vn\\
 		truonghoangle@gmail.com}
\thanks{ 
\endgraf
\endgraf
The first author was supported by Vietnam National Foundation for Science and Technology Development (NAFOSTED) under grant number 101.04-2020.10.
The last author was supported by B2022-TNA-25.
}
\begin{abstract}
This paper shows that if $R$ is a homomorphic image of a Cohen-Macaulay local ring, then $R$-module $M$ is sequentially generalized Cohen-Macaulay if and only if the difference between Hilbert coefficients and arithmetic degrees for all distinguished parameter ideals of $M$  are bounded.
\end{abstract}

\maketitle

\section{Introduction}
 Let $(R, \fkm)$ be a commutative Noetherian local ring, where $\fkm$ is the maximal ideal.
Let $M$ be a finitely generated $R$-module of dimension $d$. 
For an $\fkm$-primary ideal $I$ of $R$, it is well-known  that there are integers $\{\e_i(I;M)\}_{i=0}^d$, called the {\it Hilbert coefficients} of $M$ with respect to $I$, such that 
 	\begin{eqnarray*}
 		\ell_R(M/{I^{n+1}}M)=\e_0(I;M) \binom{n+d}{d}-\e_1(I;M) \binom{n+d-1}{d-1}+\cdots+(-1)^d \e_d(I;M)
 	\end{eqnarray*}
 	for  all $n \gg 0$. Here $\ell_R(N)$ denotes the length of an $R$-module $N$. 
 In particular, the leading coefficient  $e_0(I; M)$ is said to be the multiplicity of $M$ with respect to $I$ and $e_1(I;M)$ is called by Vasconselos (\cite{V2}) the Chern coefficient of $M$  with respect to $I$. 
 In 2008, Vasconcelos posed {\it the Vanishing Conjecture}:  $M$ is Cohen-Macaulay if and only if $\e_1(\q; M) = 0$ for some parameter ideal $\q$ of $M$.
It is shown that the relation between Cohen-Macaulayness and the Chern number of parameter ideals is quite surprising.  Motivated by some profound results of \cite{CGT,TaT20} and also by the fact that this is true for $M$ is unmixed as shown in \cite{GGHOPV}, it was asked whether the characterization of many classes of non-unmixed rings such as Buchsbaum rings, generalized Cohen-Macaulay rings, sequentially Cohen-Macaulay rings in terms of the Hilbert coefficients and other invariants of $M$ (see \cite{GO,CGT,Tr14,Tr17,OTY,TY}). The aim of our paper is to continue this research direction.

To state the results of this paper, first of all let us fix our notation and terminology.  First, a filtration $$\calD: M = D_0\supset  D_1\supset \ldots\supset D_t=W $$ of $R$-submodules of $M$ is  called the {\it dimension filtration} of $M$,  if  for all $1 \le i \le \ell$, $D_{i}$ is the largest submodule of $D_{i-1}$ with $\dim_RD_{i}<\dim_R D_{i-1}$, where $\dim_R(0) = - \infty$ for convention. We say that $M$ is {\it sequentially (generalized) Cohen-Macaulay}, 
if  $C_i=D_i/D_{i+1}$   is (generalized) Cohen-Macaulay for all $0 \le i \le \ell-1$.  A system of parameters $\underline x=x_1,x_2, \ldots, x_d$  of $M$ is said to be {\it distinguished}, if  $(x_j \mid d_{i} < j \le d)D_i=(0)$ for all $0\le i\le \ell$, where $d_i=\dim_R D_i$ (\cite[Definition 2.5]{Sch}). A parameter ideal $\q$ of $M$ is called  {\it distinguished}, if there exists a distinguished system $x_1,x_2, \ldots, x_d$ of parameters of $M$ such that $\q=(x_1,x_2, \ldots, x_d)$.
For each $i = 1, \ldots, s$, set 
$$\Lambda_{i}(M) = \{(-1)^ie_i(\fkq,M)-\adeg_i(\q;M) \mid \fkq \text{ is a distinguished parameter ideal of }  M \}, $$
where  $\adeg_i(I;M)=\sum\limits_{\begin{subarray}{l} 
  \fkp\in\Ass(M),\\ 
  \dim R/\fkp=i
\end{subarray}} \ell_{R_{\fkp}}(H^0_{\fkp R_{\fkp}}(M_{\fkp}))e_0(I;R/\fkp)$ is the {\it $i$-th arithmetic degree} of $M$ with respect to $I$ (see \cite{BM},\cite{V},\cite{V1}).
Then we have the following results as in Table \ref{table1}.
\begin{table}[h!]
\centering
		\begin{tabular}{ | l | c | c| c|c|c|c|c|c|}
		\hline
	$\Lambda_{1}(M) \subseteq (- \infty, 0]$ & $M$ & \cite{CGT}  \\ \hline 
	$ 0 \in \Lambda_{1}(M)$,  ($\star$)&  $M$ is Cohen-Macaulay     & \cite{GGHOPV}    \\ \hline
	$\left| \Lambda_{1}(M) \right| < \infty $, ($\star$) &  $M$ is generalized Cohen-macaulay & \cite{GGHOPV,GO}    \\ \hline
	$0 \in \Lambda_{i}(M) $ for all $i = 1, \ldots, d$  & $M$ is sequentially Cohen-Macaulay  &  \cite{CGT}  \\ \hline
	\end{tabular}
\caption{ Properties of a finitely generated module $M$ are carried by the behavior of the specific set. A symbol ($\star$) requires that the module $M$ be unmixed. }
\label{table1}
\end{table}

This paper aims to extend these results in the sequentially generalized Cohen-Macaulay case. The answer is affirmative, which we are eager to report in the present writing.

\medskip

\begin{theorem}[Theorem \ref{thmain}]\label{thm0}
Assume that $R$ is a homomorphic image of a Cohen-Macaulay local ring. Then the following statements are equivalent.
	\begin{enumerate}[$i)$] 
		\item$M$ sequentially generalized Cohen-Macaulay.
		
		\item The set $\Lambda_i(M)$ is finite for all $1 \le i \le d$.
		
	\end{enumerate}
\end{theorem}

The paper is divided into four sections. The next section presents some preliminaries. In Section 3, we prove that if $M$ is a sequentially generalized Cohen-Macaulay module, then the set $\Lambda_i(M)$ is finite for all $ i=1\ldots, d$. A characterization of sequentially generalized Cohen-Macaulay modules by the finiteness of  $\Lambda_i(M)$ will be shown in the last section.
\section{Preliminaries}
 In what follows, throughout this paper, let $(R, \fkm, k)$ be a Noetherian local ring, where $\fkm$ is the maximal ideal and $k = R/\fkm$
 is the infinite residue field of $R$. Suppose that $R$ is a homomorphic image of a Cohen-Macaulay local ring. Let
 $M$ be a finitely generated $R$-module of dimension $d$.

\begin{definition}[\cite{CN},\cite{CC1}, \cite{Sch}] 
\begin{enumerate}[$i)$]
\item We say that a finite filtration of submodules of $M$ $$\calF:M=M_0 \supset M_{1} \supset\ldots\supset M_s$$ {\it satisfies the dimension condition} if $\dim M_i >\dim
M_{i+1} $, for all $i=0,\ldots, s-1$ and we say this case that the filtration $\calF$ has the length $s$.\\ 
\item A filtration  $$\calD: M = D_0\supset  D_1\supset \ldots\supset D_t=W $$ of
 submodules of $M$ is said to be the {\it dimension filtration}  if $D_i$ is the largest submodule of $D_{i-1}$ with $\dim D_i < \dim D_{i-1}$ for all $i = 1,\ldots,t.$ 
Note that the dimension filtration always exists uniquely (see \cite{CN}).
\end{enumerate}
\end{definition}

\begin{notation}
	\begin{enumerate}[$\bullet$]
	\item $t$ : the length of the dimension filtration of $M$, 
	\item $\calD=\{D_i\}_{i=0}^t$ : the dimension filtration of $M$,
	\item $d_i=\dim D_i$ for all $i=0,\ldots,t$,
	\item $\calF=\{M_i\}_{i=0}^s$ : a filtration of submodules of $M$ of length $s$ satisfying the dimension condition,
	\item $\calF/x\calF=\{ ({M}_i+xM)/xM\}_{i=0}^k$ : a filtration of submodules of $M/xM$, where $x$ is a parameter element of $M$ and
		$k=\left\{ \begin{gathered}
 s-1 \text{ if } \dim(M_{s-1})=1, \hfill \\
  s \quad\quad\text{ otherwise},\hfill \\ 
\end{gathered}  \right.$
	\item $\calF(M)=\{\calF=\{M_i\}_{i=0}^t\mid \ell(D_i/M_i)<\infty \text{ for all } i=0,\ldots,t \}$.

	\end{enumerate} 
	\end{notation}

\begin{definition} 
$i)$ A system of parameters $x_1,\ldots,x_d$ of $M$
is called  a {\it distinguished system of parameters with respect to} $\calF$  if $(x_{\dim M_i+1},\ldots,x_d)M_i=0$ for all $i=1,\ldots,s$. A distinguished system of parameters of $M$ with respect to $\calD$ is simply called a {\it distinguished system of parameters} of $M$. An ideal $\Q$ is said to be a {\it distinguished parameter ideal of $M$ with respect to}  $\calF$  if it is generated by a  distinguished system of parameters of $M$ with respect to $\calF$. A distinguished parameter ideal of $M$ with respect to  $\calD$ is simply called a {\it distinguished parameter ideal} of $M$ (\cite{Sch}).

$ii)$ A system of parameters $x_1,\ldots,x_d$ is called \textit{good system of parameters} of $M$ if 
$(x_{d_i+1},\ldots,x_d)M \cap D_i=0$ for all $i=1,\ldots,t$.  An ideal $\Q$ is said to be a {\it good parameter ideal of $M$}   if it is generated by a  good system of parameters of $M$  (see \cite{CC1}).

\end{definition}

Recall that there always exists distinguished systems of parameters of $M$ with respect to $\calF$ (see \cite[Lemma 2.6]{Sch}). 
Note that, if $x_1,x_2,\ldots,x_d$ is a distinguished system of parameters of $M$ with respect to $\calF$, then $x_2,\ldots,x_d$  is also a distinguished system of parameters of $M/x_1M$  with respect to $\calF/{x_1}\calF$.
Clearly, a good system of parameters is a distinguished system of parameters.

\begin{definition}[\cite{BM},\cite{V},\cite{V1}] Let $I$ be an $\m$-primary ideal of $R$.
The {\it $i$-th arithmetic degree} of $M$ with respect to $I$ is defined by
$$\adeg_j(I;M)=\sum\limits_{\fkp\in\Ass(M),\:\dim R/\fkp=j}\ell(H^0_{\fkp R_{\fkp}}(M_{\fkp}))e_0(I;R/\fkp).$$
 
\end{definition}

The following result is an immediate consequence of Proposition 3.2 in \cite{CGT}.
\begin{lemma}[{\rm c.f. \cite[Proposition 3.2]{CGT}}]\label{lm2}
Let $I$ be an $\m$-primary ideal of $R$ and $\calF\in \calF(M)$. Then 
$$\adeg_j(I;M)=\left\{ \begin{gathered}
  \ell(H^0_{\m}(M))\quad\,\,\text{ if } i=0, \hfill \\
  e_0(I;M_i)\quad\quad \text{ if  }  d_i=j \text{ for some } i=0,\ldots,t-1,  \hfill \\
  0 \quad\quad\quad\quad\quad\,\,\text{ otherwise}. \hfill \\ 
\end{gathered}  \right.$$
\end{lemma}

\begin{lemma}\label{lmf-a} 
Let  $\calF\in\calF(M)$ and $\q$ a parameter ideal of $M$. Suppose that there exists a filter regular element $x\in\q$ of $M$ 
such that $\calF/x\calF\in \calF (M/xM)$. We have $$\adeg_j(\Q;M/xM)=\adeg_{j+1}(\Q;M),$$ for all $j\geq 1$. Moreover, if $\q:=(x:=x_1,x_2,\ldots,x_d)$ is a distinguished parameter ideal of $M$ with respect to $\calF$, then $\adeg_0(\Q;M/xM)\geq \adeg_1(\Q;M).$
\end{lemma}
\begin{proof}

Since $\calF/x\calF\in\calF(M/xM)$, 
we have 
 $\calD/xM\in\calF(M/xM)$. Moreover, the length of $\calD/xM$ is   $t-1$ if $d_{t-1}=1$ and $t$ otherwise. Since $x$ is a filter regular element of $M$, $x$ is a regular element of $M/D_i$ for all $i=1,\ldots,t$. 
Thus we have $xM\cap D_i=xD_i$. By Lemma \ref{lm2} and $\calD/xM \in \calF(M/xM)$, we have
$$\begin{aligned}
\adeg_{d_i-1}(\q ;M/xM)&=e_0(\q;(D_i +xM)/xM)\\
&=e_0(\q ;D_i /xD_i)=e_0(\q ;D_i)=\adeg_{d_i}(\q ;M)
\end{aligned}$$
and $\adeg_j(\q;M)=0=\adeg_{j-1}(\q;M/xM)$
for all $i=0,\ldots,t-1$ and $2\le d_i<j<d_{i-1}$.  Hence $$\adeg_j(\Q;M/xM)=\adeg_{j+1}(\Q;M),$$ for all $j\geq 1$.

Now, if $d_{t-1}>1$, then $\adeg_1(\q;M)=0\leq \adeg_0(\q;M)$. So we can assume  that $d_{t-1}=1$. Since $\q=(x,x_2,\ldots,x_d)$ is a distinguished ideal of $M$ with respect to $\calF\in\calF(M)$, by Lemma \ref{lm2} we get
$$
\begin{aligned}
\adeg_1(\q;M)&=e_0(\q; M_{t-1})=e_0(x;M_{t-1})
=e_0(x;D_{t-1})\leq \ell(D_{t-1}/xD_{t-1}).
\end{aligned}$$
Since $H^0_{\m}(M/xM) \supseteq (D_{t-1}+xM)/xM\cong D_{t-1}/(D_{t-1}\cap xM)=D_{t-1}/xD_{t-1}$, we have
  $$\adeg_1(\q;M)\leq \ell(D_{t-1}/xD_{t-1})\leq \ell(H^0_{\m}(M/xM)) = \adeg_0(\q;M/xM),$$
as required.
\end{proof}

The following lemma was proved by \cite[Lemma 3.3, Lemma 3.4]{CGT}.
\begin{lemma}\label{lms-e} The following statements are true. \\
$(i)$ Assume that $d\geq 2$. Let $x$  be a superficial element of $M$ for a parameter ideal $\q$ of $M$. Then
$$e_{i}(\q;M)=e_{i}(\q;M/xM)$$
for all $i=0,\ldots,d-2$ and $(-1)^{d-1}e_{d-1}(\q;M)=(-1)^{d-1}e_{d-1}(\q;M/xM)+\ell(0:_Mx)$.\\
$(ii)$ Let $N$ be a submodule of $M$ with $\dim N = s < d$ and $I$ an $\m$-primary ideal of $R$. Then
$$e_j(I;M)= \left\{ \begin{gathered}
  e_j(I;M/N) ,\text{ if } 0\leq j\leq d-s-1, \hfill \\
  e_{d-s}(I;M/N) +(-1)^{d-s}e_0(I,N) \text{ if } j=d-s. \hfill \\ 
\end{gathered}  \right.$$
\end{lemma}

\section{The finiteness of the set $\calP_{\calD}(M)$}

Recall that 
the definition of a sequentially Cohen-Macaulay module was introduced first by LT. Nhan and the first author (\cite{CN}).
\begin{definition} \label {def1}
A filtration of submodules $ \calF=\{M_i\}_{i=0}^t$ of $M$ is called a {\it generalized Cohen-Macaulay filtration} 
if $\calF \in \calF(M)$ and $M_{i-1}/M_i$ are generalized Cohen-Macaulay modules for all $i=1,\ldots,t-1$. 
A module $M$ is called {\it sequentially generalized Cohen-Macaulay} if it has a generalized Cohen-Macaulay filtration. 
In particular, $\calF(M)$ is the set of all generalized Cohen-Macaulay filtrations of $M$. If $M$ is a sequentially Cohen-Macaulay module, then  $M$ is a sequentially generalized Cohen-Macaulay module.\\
\end{definition}
 Now, the function
$$H^{ad}_{I,M}(n)=\ell(M/I^{n+1}M)-\sum\limits_{i=0}^{d}\adeg_i(I;M)\binom{n+i}{i}$$
is called an {\it adjusted Hilbert-Samuel function} of $M$ with respect to $I$.  
It is well known that $\ell(M/I^{n+1}M)$ becomes a polynomial for large enough $n>0$. So the function $H^{ad}_{I,M}(n)$ become a polynomial $P^{ad}_{\q,M}(n)$. Such polynomial is called {\it adjusted Hilbert-Samuel polynomial} and of the form
$$P^{ad}_{I,M}(n)=\sum\limits_{i=1}^{d}\left((-1)^ie_i(I;M)-\adeg_{d-i}(I;M)\right)\binom{n+d-i}{d-i}.$$
These integers $a_i(I;M)=(-1)^ie_i(I;M)-\adeg_{d-i}(I;M)$ are called \textit{adjusted Hilbert coefficients} of $M$ with respect to $I$ for all $i=1,\ldots,d$.
We denote by $\calP_{\calF}(M)$ the set of all adjusted Hilbert-Samuel polynomials $P^{ad}_{\Q,M}(n)$, 
where $\q$ runs over the set of all distinguished parameter ideals of $M$ with respect to $\calF$. 

Recall that a system $x_1,\ldots,x_m$ in $R$ is said to be d-{\it sequence} on $M$ (see \cite{H}, \cite{T}) if 
$$(x_1,\ldots,x_{i-1})M:x_{i}x_k=(x_1,\ldots,x_{i-1})M:x_k$$
for all $i=1,\ldots,m$ and $k\geq i$. 
The sequence $\underline x$ is said to be dd-{\it sequence} on $M$ (see \cite{CC3}) if $x_1^{n_1},\ldots,x_s^{n_s}$ is a d-sequence on $M$ and  $x^{n_1}_1,\ldots,x^{n_i}_i$ is a d-sequence on $M/(x_{i+1}^{n_{i+1}},\ldots,x_s^{n_s})M$ for all positive integers $n_1,\ldots,n_s$ and all $i=1,\ldots,s-1$. According to D. T. Cuong and the first author, if the parameter ideal $\q$ is generated by a dd-sequence on $M$, the adjusted Hilbert coefficients are described as follows.

\begin{lemma}[{\cite[Theorem 6.2]{CC2}}]\label{f-dd}
Let $M$ be a sequentially generalized Cohen-Macaulay module and $\q=(x_1,\ldots,x_d)$ a system of parameters of $M$. Assume that $x_1,\ldots,x_d$ is a dd-sequence on $M$. Then the adjusted Hilbert coefficients are of the form
$$ a_{d-d_k}(\q;M)=\sum\limits_{j=1}^{d_k} \binom{d_k-1}{j-1}\ell(H^j_{\m}(M/D_k))$$
for $k=0,\ldots,t-1$,
$$ a_{d-i}(\q;M)=\sum\limits_{j=1}^{i} \binom{i-1}{j-1}\ell(H^j_{\m}(M/D_k))$$
for $d_k<i<d_{k-1}$ and $a_d(\q;M)=0$.
\end{lemma}

Now with the above notations, we have the main result in this section.
\begin{theorem}\label{thm1}
Let $M$ be a sequentially generalized Cohen-Macaulay module and $\calF\in \calF(M)$. Then the set $\calP_{\calF}(M)$ of adjusted Hilbert-Samuel polynomials is finite.
\end{theorem}

\begin{Setting}
In this section, from now on, we assume that
$M$ is a sequentially generalized Cohen-Macaulay module. 
Set $W=H^0_\fkm(M)$, $\overline M= M/W$, and $\overline N=(N+W)/W$ for all submodules $N$ of $M$. Let $\calF\in\calF(M)$, $\q:=(x:=x_1,x_2,\ldots,x_d)$ be a
distinguished parameter ideal of $M$ with respect to $\calF$.\end{Setting}
Set
 $$I(M)=\sup\{\ell(M/\q M)-e_0(\q ;M)\mid \q \text { is a parameter ideal of } M \}.$$
and $I(\calF,M)=\displaystyle\sum \limits^{t-1}_{i=0}I(M_i/M_{i+1})+\ell(M_t)$. 

\begin{fact} \label{rm_inv} With this notation, we have
\begin{enumerate}[$i)$] 
\item The filtration of submodules
 $\overline{\calF}:=\{ \overline{M_i}\}_{i=0}^t$ of $\overline M$ is a generalized Cohen-Macaulay  filtration of $\overline M$ and $$I({\calF},M) = I(\overline{\calF},\overline M) +\ell(W)$$
(\cite[Lemma 6]{CLT}).
\item The module $M/xM$ is  sequentially generalized Cohen-Macaulay and $\calF/x\calF\in\calF(M/xM)$ (\cite[Corollary 2]{CLT}).
Moreover, we have
$$I(\calF/x\calF,M/xM)\leq I(\calF,M)$$
(\cite[Lemma 4]{CLT}).
 \end{enumerate}

\end{fact}
Now, let $\calS=\bigoplus \limits_{n\geq  0}S_n $ be a standard Noetherian graded ring and
$E=\bigoplus \limits_{n \in \Bbb Z}E_n $ a finitely generated graded $\calS$-module. The {\it Castelnouvo-Mumford regularity} reg$(E)$ of $E$ is defined by
$$\reg(E)=\sup\{n+i\mid [H^i_{S_+}(E)]_n\ne0, i\geq 0 \}$$
and simply called {\it regularity}, where $S_+=\bigoplus \limits_{n> 0}S_n$. Let $N$ be a finitely generated $R$-module and $Q$ a parameter ideal of $N$. We always denote the associate graded module of $N$ with respect to $Q$ by $G_{Q}(N)$, i.e $G_{Q}(N)=\bigoplus \limits_{n\geq 0}{Q^nN/Q^{n+1}N}$.
With this notation, we have 

\begin{fact}\label{lm11}\label{fucnega}
\begin{enumerate}[$i)$] 
\item There is a constant $C=C_{\calF}$ such that
$$\reg(G_{\Q}(M))\leq C=(3I(\calF,M))^{d!}-2I(\calF,M),$$
 for all distinguished parameter ideals $\Q$ of $M$ with respect to $\calF$ ({\cite[Theorem 4]{CLT}}). 
\item We have $H^{ad}_{\Q,M}(n)\geq 0$ for all $$n\geq \reg(G_{\q}(M))+\binom {\reg(G_{\q}(M)) +d-1} {d-1} I(\calF, M)+d,$$
({\cite[Theorem 4.4]{Lo}}).
\end{enumerate}

\end{fact}
The following result give an upper bound for the adjusted Hilbert-Samuel function.
\begin{lemma}\label{lm31}  We have
$$H^{ad}_{\Q,M}(n)\leq \sum \limits^{t-1}_{i=0}\binom{n + d _i-1} {d_i -1}I(M_i/M_{i+1})+\ell(M_t)-\ell(W)),$$
for all $n\geq 0$,
\end{lemma}
\begin{proof} 
Note that from the following exact sequence
$$0\rightarrow (\Q^{n+1}M\cap M_1)/{\Q^{n+1}M_1 \longrightarrow M_1/\Q^{n+1}M_1
\longrightarrow M/\Q^{n+1}M \longrightarrow }
M/\Q^{n+1}M+M_1\rightarrow 0,$$
we get $\ell (M/\Q^{n+1}M)\leq \ell(M/\Q^{n+1}M+M_1)+\ell(M_1/\Q^{n+1}M_1) $. 

Now we argue by the induction on the length $t$ of the dimension filtration of $M$.
The case $t=1,$   it follows from  $\q M_1=0$ and Lemma 1.1 in \cite{LT} that we have
$$\begin{aligned} \ell(M/\Q^{n+1}M)&\leq \ell(M/\Q^{n+1}M+M_1)+\ell(M_1)\\
&\leq \binom{n + d }{d} e_0(\Q;M) +\binom{n + d -1}{d -1}I(M/M_1)+\ell(M_1).
\end{aligned}$$
for all $n\geq 0$.

Now, assume that $t>1$ and that our assertion holds true for $t - 1$. By the inductive hypothesis, we have 
 $$\begin{aligned}\ell (M/\Q^{n+1}M)&\leq \ell(M/\Q^{n+1}M+M_1)+\ell(M_1/\Q^{n+1}M_1) \\
&\leq \binom{n + d }{d}e_0(\Q;M_0)+\binom{n + d -1}{d-1}I(M/M_1)+\ell(M_1/\Q^{n+1}M_1)\\
&\leq \sum \limits^{t-1}_{i=0}\binom{n + d _i}{d_i}e_0(\q;M_i) +\sum \limits^{t-1}_{i=0}\binom{n + d _i-1} {d_i -1}I({M}_i/M_{i+1})+\ell(M_t),\end{aligned}$$
for all $n\geq 0$.
By Lemma \ref{lm2}, we have $$H^{ad}_{\Q,M}(n)\leq \displaystyle \sum \limits^{t-1}_{i=0}\binom{n + d _i-1} {d_i -1}I({M}_i/M_{i+1})+\ell(M_t)-\ell(W),$$ for all $n\geq 0$, as requested. 
\end{proof}

\begin{Theorem}\label{claim39}
     Let  $C=C_{\calF}=(3I(\calF,M))^{d!}-2I(\calF,M)$ as in Fact \ref{lm11} i). Then we have
\begin{enumerate}[$i)$]
\item $\mid e_1(\Q;M)+\adeg_{d-1}(\Q;M)\mid \leq I(M/M_1).$
\item $\mid (-1)^ie_i(\Q;M)-\adeg_{d-i}(\Q;M)\mid \leq 2^{i-1}\left((C+1)^{d-1}I(\calF,M)+d+C+2\right)^{i-1}I(\calF,M)$ for all $i=2,\ldots,d-1.$
\item $\mid e_d(\Q;M)\mid \leq 2^{d-1}\left((C+1)^{d-1}I(\calF,M)+d+C+2\right)^{d-1}I(\calF,M).$
\end{enumerate}
\end{Theorem}

\begin{proof}

$i)$ By Fact \ref{fucnega} ii), we have $P^{ad}_{\q,M}(n)=H^{ad}_{\q,M}(n)\ge 0$ for all $n\gg 0$. Thus by Lemma \ref{lm31}, we have  
$$
\begin{aligned}
0\leq P^{ad}_{\q,M}(n) &=\binom{n+d-1}{d-1}( (-1)^{d-1}e_1(\q;M)-\adeg_{d-1}(\q;M))+ \text { lower terms }\\ 
& \leq \binom{n+d-1}{d-1}I(M/M_1)+ \text { lower terms }.
\end{aligned}$$
Therefore we have $$0\leq (-1)e_1(\q;M)-\adeg_{d-1}(\q;M)\leq I(M/M_1),$$
as requested.

 $ii)$ and $iii)$. Now we proceed by induction on $d$ to show that $ii)$ and $iii)$. Set $r=\reg(G_{\q}(M))$. 
 The case $d=2$,  we have
$$e_2(\Q;M)=H^{ad}_{\Q,M}(n)+\ell(H^0_{\m}(M))+(e_1(\Q;M)+\adeg_1(\Q;M))\binom{n+1}{1}$$
for all $n\geq r+1$.
Thus by Lemma \ref{lm31} and Theorem \ref{fucnega}  we have 
$$\begin{aligned}\mid e_2(\Q;M)\mid &\leq \sum \limits^{t}_{i=0}\binom{n + d _i-1} {d_i -1}I({M}_i/M_{i+1})+I(M/M_1)(n+1)\\
&\leq 2(n+1)^{2-1}I(\calF,M).
\end{aligned}$$
for all $n\geq r+\binom {r +1} {1} I(\calF, M)+d$. \\
Now choose $n=(C+1)^{2-1}I(\calF, M)+d+C+1$. Then $n\geq r+\binom {r +1} {1} I(\calF, M)+d$ by Fact \ref{lm11} i). 
Therefore we have
$$\begin{aligned}\mid e_2(\Q;M)\mid &\leq 2^{2-1}\left((C+1)^{2-1}I(\calF, M)+d+C+1+1\right)I(\calF,M)\\
&= 2^{d-1}\left((C+1)^{d-1}I(\calF, M)+d+C+2\right)^{d-1}I(\calF,M),
\end{aligned}$$
as required.

Now suppose that $d>2$ and that our assertion holds true for $d-1$.
We have by Lemma \ref{lms-e}$(ii)$ that 
$$e_i(\Q;M)=e_i(\Q;\overline M)$$ for all $i=0,\ldots,d-1$ and
$(-1)^de_d(\Q;M)=(-1)^de_d(\Q;\overline M)+\ell(W)$. Moreover, by Fact \ref{rm_inv}$(i)$ we have$$I(\overline\calF,\overline M)+\ell(W)=I(\calF,M).$$
Consequently, $C_{\calF}\geq C_{\overline\calF}$.
Therefore we can assume $W=0$.
Recall that $\underline{x}=x_1,\ldots,x_d$ is a distinguished system of parameters of $M$ with respect to $\calF$ such that $\q=(\underline{x})$ and $x=x_1$ is a superficial element of $M$ for $\q$.
Hence $x$ is a regular element of $M$ since $W=0$. So by Lemma \ref{lms-e}$(i)$,
\begin{flushright}
$e_i(\Q;M)=e_i(\Q;M/xM) \hspace{5cm}(1)$
\end{flushright}
for all $i=0,\ldots,d-1$.
On the other hand, it follows from Lemma \ref{lmf-a} and Fact \ref{rm_inv} $ ii)$ that
we have
\begin{flushright}
$\adeg_{d-i}(\Q;M)=\adeg_{d-1-i}(\Q;M/xM) \hspace{3cm}(2)$
\end{flushright} 
 for all $i=0,\ldots,d-2$
 and
\begin{flushright}$I(\calF/x\calF,M/xM)\leq I(\calF,M).\hspace{5cm}(3)$ \end{flushright}
Set $C_x=C_{\calF/x\calF}$. It follows that
\begin{flushright}$C_x=C_{\calF/x\calF}\leq C_{\calF}=C.\hspace{5cm}(4)$ \end{flushright}
Note that $x_2,\ldots,x_d$ is a distinguished system of parameters of sequentially generalized Cohen-Macaulay  $M/xM$ with respect to $\calF/x\calF\in \calF(M/xM)$.
Therefore it follows from the inductive hypothesis and $(1)-(4)$ that we have
$$\begin{aligned}\mid (-1)^ie_i(\Q;M)&- \adeg_{d-i}(\Q;M)\mid\\
&=\mid (-1)^ie_i(\Q;M/xM)- \adeg_{d-1-i}(\Q;M/xM)\mid\\
&\leq 2^{i-1}\left((C_x+1)^{d-2}I(\calF/x\calF,M/xM)+(d-1)+C_x+2\right)^{i-1}I(\calF/x\calF,M/xM)\\
&\leq 2^{i-1}\left((C+1)^{d-1}I(\calF,M)+d+C+2+I(\calF,M)\right)^{i-1}I(\calF,M)
\end{aligned}$$
for all $i=2,\ldots,d-2$.  

Fact, if $d_{t-1}>1$ then $\adeg_1(\q;M)=0$ by Lemma \ref{lm2}. So
$$\begin{aligned} \mid (-1)^{d-1}&e_{d-1}(\Q;M)-\adeg_1(\Q;M)\mid\\
&=\mid (-1)^{d-1}e_{d-1}(\Q;M/xM)\mid \\
&\leq 2^{(d-1)-1}\left((C_x+1)^{d-2}I(\calF/x\calF,M/xM)+(d-1)+C_x+2\right)^{(d-1)-1}I(\calF/x\calF,M/xM)\\
&\leq 2^{d-1}\left((C+1)^{d-1}I(\calF,M)+d+C+2\right)^{d-1}I(\calF,M).
\end{aligned}$$
The last inequality is followed by $(3)$ and $(4)$.

Now we can assume that $d_{t-1}=1$. Set $\widetilde M =M/D_{t-1}$ and  $\widetilde N=(N+D_{t-1})/D_{t-1}$ for all submodules $N$ of $M$.
Then  $\widetilde M$ is sequentially generalized Cohen-Macaulay, 
$\widetilde{\calF}=\{\widetilde {M_i}\}_{i=0}^{t-1} \in\calF(\widetilde {M})$ and $\q$ is also a distinguished parameter ideal of $\widetilde{M}$ with respect to $\widetilde {\calF}$.
Note that $x_1,\ldots,x_d$ is a distinguished system of parameters of sequentially generalized Cohen-Macaulay  $\widetilde M $ with respect to $\widetilde \calF$.
Using similar above arguments, it follows from $H^0_\fkm(\widetilde M)=0$ that we show that results $1)-4)$ for module $\widetilde M$.

  By Fact \ref{lm11} i), we have $$ \reg(G_{\q}(\widetilde{M}))\leq \widetilde C:=(3I(\widetilde {\calF},\widetilde {M}))^{d!}-2I(\widetilde{\calF},\widetilde {M}).$$ 
Since  
$(M_i\cap D_{t-1})/(M_{i+1}\cap D_{t-1})=(M_{i+1}+M_i\cap D_{t-1})/M_{i+1}$ is a submodule of the module $D_{t-1}/M_{t-1}$ of finite length for all $i=0,...,t-2$, we have
$$\begin{aligned}
I(\calF,M)=&\sum \limits^{t-2}_{i=0}I(M_i/M_{i+1})+I(M_{t-2}/M_{t})+\ell(M_t)\\
&\geq\sum \limits^{t-2}_{i=0}\left(I(M_i/M_{i+1})-l((M_i\cap D_{t-1})/(M_{i+1}\cap D_{t-1}))\right)+I(M_{t-1})\\
&=\sum \limits^{t-2}_{i=0}I((M_i+D_{t-1})/(M_{i+1}+D_{t-1}))+I(M_{t-1})\\
&=I(\widetilde{\calF},\widetilde{M})+I(M_{t-1}).
\end{aligned}$$
This implies that $\widetilde{C} \leq C $. 
We have by Lemma \ref{lms-e}$(ii)$ and Lemma \ref{lm2} that 
$$(-1)^{d-1}e_{d-1}(\q,M)-\adeg_1(\q,M)=(-1)^{d-1}e_{d-1}(\q,\widetilde M).$$
Therefore it follows from the inductive hypothesis and similar results $1)-4)$ for module $\widetilde M$  that we have
$$\begin{aligned} \mid (-1)^{d-1}e_{d-1}(\Q;M)-&\adeg_1(\Q;M)\mid=\mid (-1)^{d-1}e_{d-1}(\Q;\widetilde M)\mid=\mid (-1)^{d-1}e_{d-1}(\Q;\widetilde M/x \widetilde M)\mid \\
&\leq 2^{d-2}\left((\widetilde C_x+1)^{d-1}I(\widetilde\calF/x\widetilde\calF,\widetilde M/x\widetilde M)+d-1+\widetilde C_x+2\right)^{d-1}I(\widetilde\calF/x\widetilde\calF,\widetilde M/x\widetilde M)\\
&\leq 2^{d-1}\left((\widetilde C+1)^{d-1}I(\widetilde{\calF},\widetilde M)+d+\widetilde C+2\right)^{d-1}I(\widetilde{\calF},\widetilde M)\\
&\leq 2^{d-1}\left((C+1)^{d-1}I(\calF,M)+d+C+2\right)^{d-1}I(\calF,M)
\end{aligned}$$
where $\widetilde C_x=C_{\widetilde\calF/x\widetilde\calF}$.

Now,  we have
 $$ (-1)^de_d(\Q;M)=H^{ad}_{\Q,M}(n)- \sum\limits_{i=1}^{d-1}\left((-1)^ie_i(\Q;M)-\adeg_{d-i}(\Q;M)\right)\binom{n+d-i}{d-i}, $$
 for all $n\geq r+1$. 
Furthermore, for all $n\geq r+\binom {r +1} {1} I(\calF, M)+d$, by Lemma \ref{lm31} and Fact \ref{fucnega}, we have 
 $$\begin{aligned}\mid e_d(\Q;M)\mid \leq &\sum \limits^{t-1}_{i=0}\binom{n + d _i-1} {d_i -1}I({M}_i/M_{i+1})+\mid (-1)e_1(\Q;M)-\adeg_{d-1}(\Q;M)\mid \binom{n+d-1}{d-1}\\
&+\sum\limits_{i=2}^{d-1}\mid (-1)^ie_i(\Q;M)-\adeg_{d-i}(\Q;M)\mid \binom{n+d-i}{d-i}\\
&\leq \sum \limits^{t-1}_{i=0}(n+1)^{d-1}I({M}_i/M_{i+1})+(n+1)^{d-1}I(\calF,M)\\
&+\sum\limits_{i=2}^{d-1} 2^{i-1}\left((C+1)^{d-1}I(\calF,M)+d+C+2\right)^{i-1}I(\calF,M)(n+1)^{d-i},
\end{aligned}$$
where the last inequality is followed by $i)$, $ii)$. 
Choose $n=(C+1)^{d-1}I(\calF, M)+d+C+1$ and note that $n\geq r+\binom {r +d-1} {d-1} I(\calF, M)+d$. Then we have
$$\begin{aligned}\mid e_d(\Q;M)\mid &\leq 2\left((C+1)^{d-1}I(\calF,M)+d+C+2\right)^{d-1}I(\calF,M)\\
&+\sum\limits_{i=2}^{d-1} 2^{i-1}\left((C+1)^{d-1}I(\calF,M)+d+C+2\right)^{d-1}I(\calF,M)\\
&=2^{d-1}\left((C+1)^{d-1}I(\calF,M)+d+C+2\right)^{d-1}I(\calF,M).
\end{aligned}$$
as required.
\end{proof}

Now, we are in a position to prove the main theorem in this section.
\begin{proof}[Proof of Theorem \ref{thm1}] 
This is now immediately seen from Theorem \ref{claim39}. 
\end{proof}

We close this section with the following example, which shows that the condition $\calF\in \calF(M)$ in Theorem \ref{thm1} and  \ref{claim39}  can not be omitted whenever $M$ is sequentially Cohen-Macaulay.
\begin{example}\label{ex1}
Let $R= k[[X, Y,Z]]$ be the formal power series ring over a field $k$. Let 
$M=k[[X, Y,Z]] \oplus (k[[X, Y,Z]]/(Z^2))$
be $R$-module. For an integer $m\ge 1$, set $\q_m=(X^m,Y^m,Z)$.
Then we have the following.
\begin{enumerate}[$i)$] 
\item $M$ is sequentially Cohen-Macaulay of dimension $2$.
\item $\q$ is  a parameter ideal of $M$ but $\q$ is not a distinguished parameter ideal with respect to  $F\in\calF(M)$.
\item $H^{ad}_{\Q_m,M}(n)=-m^2\binom{n+1}{1}$ for all $m,n >1$. Hence the set $\Lambda_1(M)$ is finite but the set $\calP_{\calF}(M)$ and $\Lambda_2(M)$ is infinite.
\end{enumerate}

\end{example}

\section{Characterization of the finiteness of the set $\calP_\calD(M)$}

In this section, we give a characterization of sequentially generalized Cohen-Macaulay modules. In particular, we have
\begin{theorem}\label {thmain}
Let $R$ be a homomorphic image of a Cohen-Macaulay local ring. Then the following statements are equivalent:
\begin{enumerate}[$i)$] 

\item $M$ is  sequentially generalized Cohen-Macaulay.
\item The set $\calP_{\calF}(M)$ is finite for all $\calF \in \calF(M)$.
\item The set $\calP_{\calF}(M)$ is finite for some $\calF \in \calF(M)$.
\item The set $\calP_{\calD}(M)$ is finite.
\end{enumerate}
\end{theorem}
\begin{Setting}\label{defNx}
In this section, from now on, we assume that $R$ is a homomorphic image of a Cohen-Macaulay local ring.
 Let $\calF=\{M_i\}_{i=0}^t\in\calF(M)$ and $\underline x =x_1,\ldots,x_d$ be a distinguished system of parameters of $M$. 

\end{Setting}

\begin{notation}\label{defNx}
\begin{enumerate}[$\bullet$] 
\item Set $\q:=(x:=x_1,x_2,\ldots,x_d)$ and $\q_i=(x_1,\ldots,x_i)$ for all $i=1,\ldots,d$
Let $\underline n =(n_1,\ldots,n_d)\in \Bbb N^d$
$\underline x_i^{\underline n_i}=x_{i+1}^{n_{i+1}},\ldots,x_{d}^{n_{d}}$ and $\q_i^{\underline n_i}=(x_{1}^{n_{1}},\ldots,x_{i}^{n_{i}})$.
We set
$$\calN(\underline x;M)=\{ \underline n \in \Bbb N^d  \mid \underline x_i^{\underline n_i} \text{ is a distinguished system of parameters of }
M/\q_i^{\underline n_i}M \text{ for all } 0\le i\le d-1\}.$$
\item Let $d_{i+1}< k\leq d_{i}$ for some $i=0,\ldots,t-1$. Then
$\calF/\q_k\calF:=\{(M_i+\q_kM)/\q_kM\}_{i=0}^s$
is a filtration of submodules of $M$ satisfying the dimension condition. The length $s$ of a filtration of submodules $\calF/\q_k\calF$ of $M$ is $i$, if $k=d_i$ and $i+1$, otherwise. 
Furthermore, $x_{k+1},\ldots.,x_d$ is a distinguished system of parameters of $M/\q_kM$ with respect to $\calF/\q_k\calF$.
 Stipulating $x_{0}=0$ and $\calF_0=\calF$, 

\end{enumerate}
\end{notation}

\begin{lemma}\label{p1}
Suppose that $d\geq 2$ and $\underline{x}$ is a distinguished system of parameters  of $M$ such that $\calN(\underline {x};M)\ne \emptyset$. Then we have
\begin{enumerate}[$i)$] 
\item $\underline x^{\underline n}$ is a d-sequence on $M$ for all  $\underline n\in \calN(\underline x; M)$. 
\item  $(n_1,n_{2}m_{2},\ldots,n_{d}m_{d})\in \calN(\underline {x};M)$  for all $(m_{2},\ldots,m_d) \in \calN(\underline x_1^{\underline n_1}; M/x_1^{n_1}M)$.
\item If $M$ is a sequentially generalized Cohen-Macaulay module, then 
there is an $\underline n\in \calN(\underline x;M)$ such that $\underline x^{\underline n}$ is a dd-sequence on $M$.
\end{enumerate}
\end{lemma}
\begin{proof} $i)$. Let $\underline n=(n_1,\ldots,n_d) \in \calN(\underline x; M)$.  Then $\underline x_{i-1}^{\underline n_{i-1}}$ is a distinguished system of parameters of $M/\q_{i-1}^{\underline n_{i-1}}M$.
 Thus $(\underline x_{i-1}^{\underline n_{i-1}})H^0_{\m}(M/\q_{i-1}^{\underline n_{i-1}}M)=0.$ Therefore
$\underline x ^{\underline n} $ is a filter regular $M$-sequence. Moreover,
$$\q_{i-1}^{\underline n_{i-1}}M:x_{i}^{n_i}=H^0_{\fkm}(M/\q_{i-1}^{\underline n_{i-1}}M)=\bigcup\limits_{n=1}^{\infty}(\q_{i-1}^{\underline n_{i-1}}M:(\underline x^{\underline n})^n)$$
for all $i=1,\ldots,d$. So $\underline x ^{\underline n} $ is a d-sequence by \cite[Theorem 1.1$(vii)$]{T}.

$ii)$. This is now immediately seen from the definition of the distinguished system of parameters.

$iii)$. By the Artin-Rees Lemma, there are positive integers $m_1,\ldots,m_d$ such that $x_1^{m_1},\ldots,x_d^{m_d}$ is a good system of parameters of $M$.
We have by \cite[Theorem 3.8, Corollary 3.9]{CC2} that there are positive integers $k_1,\ldots,k_d$ such that
 $x_1^{n_1},\ldots,x_d^{n_d}$ is a dd-sequence on $M$, where $n_i=m_ik_i$ for all $i=1,\ldots,d$. 
We have $(n_1,\ldots,n_d)\in \calN(\underline x; M)$ by \cite[Corollary 3.7]{CC1} and the definition of dd-sequence.
\end{proof}

\begin{lemma}\label{lm41}There is a distinguished system of parameters $\underline x= x_1,\ldots,x_d$ of $M$ such that
$\calF/\q_i\calF\in \calF(M/\q_iM)$ for all $\calF\in \calF(M)$ and $i=1,\ldots,d$.
Moreover, $\calN(\underline x;M) \ne \emptyset$.
\end{lemma}
\begin{proof}
The first assertion is now immediately seen from  Lemma 2.7 in \cite{CGT}. \\
Now we will prove that $\calN(\underline x; M)\ne \emptyset$ by induction on $d$. The case $d=1$ is obvious. Assume that $d>1$. Then  $x_2,\ldots,x_d$ is a distinguished system of parameters of $M/x_1M$ with respect to $\calD/{x_1}\calD$. Since 
$\calD/{x_1}\calD\in \calF(M/x_1M)$, there are positive integers $n_2,\ldots,n_d$ such that $x_2^{n_2},\ldots,x_d^{n_d}$ is a distinguished system of parameters of $M/x_1M$. 
By the inductive hypothesis, we have $(m_2,\ldots,m_d)\in \calN(x_2^{n_2},\ldots,x_d^{n_d};M/x_1M)$ for some $m_i$. Therefore we have
$(1,n_2m_2,\ldots,n_dm_d)\in \calN(\underline x;M)$, as requested.
\end{proof}

\begin{lemma}\label{lm42}
Assume that  $d\geq 2$ and $\calF=\{M_i\}_{i=0}^t\in\calF(M)$. Then, $H^1_{\fkm}(M/M_{i+1})$ are of finite length for $i = 0, 1, . . . , t - 1$ and $d_i \geq 2$.  
\end{lemma}
\begin{proof} We first see that $\Ass_R(M_i/M_{i+1})=\Assh_R(M_i/M_{i+1})\cup \{\m\}$. It follows by \cite[Lemm 3.1]{GN} that $H^1_{\fkm}(M_i/M_{i+1})$ are of finite length for all $i=0,\ldots,t-1$ and $d_i\geq 2$.
Now we will prove by induction on $i$. If $i=0$, clearly that $H^1_{\fkm}(M/M_1)$ is of finite length. Assume that $i>0$ and that assertion holds true for $i-1$. Now the short exact sequence 
$$0 \longrightarrow M_i/M_{i+1}\longrightarrow  M/M_{i+1}\longrightarrow  M/M_i\longrightarrow 0$$
induces the long exact sequence $$\ldots\longrightarrow H^1_{\fkm}(M_i/M_{i+1})\longrightarrow H^1_{\fkm}(M/M_{i+1})\longrightarrow H^1_{\fkm}(M/M_{i})\longrightarrow \ldots.$$
If $d_{i} \geq 2$ then $d_{i-1}> d_i\geq 2$. Thus $H^1_{\fkm}(M/M_{i})$ is of finite length by  induction. So $H^1_{\fkm}(M/M_{i+1})$ is of finite length.
\end{proof}

\begin{theorem} \label{thm2} Assume that  $d\geq 2$. Suppose that there is an integer $C$ and a distinguished parameter ideal  $\q=(x_1,\ldots,x_d)$ of $M$ as in Lemma \ref{lm41} such that 
$$ \mid a_{i}(\q^{\underline n};M) \mid \leq C $$
for all $\underline n\in \calN(\underline x;M)$ and $i=0,\ldots,d-1$. Then we have $$\fkm^{C}H^j_{\fkm}(M/D_{k+1})=0$$ for all $j=1,\ldots,d_k-1, d_k\geq 2$ and $k=0,\ldots,t-1$.
\end{theorem}
\begin{proof} Let $\underline n \in \calN(\underline {x};M)$. We have by Lemma \ref{lms-e}$(ii)$ and Lemma \ref{lm2}$(i)$ that
$$e_i(\underline x ^{\underline n};M)=e_i(\underline x ^{\underline n};M/W),$$ $$\adeg_{d-i}(\underline x ^{\underline n};M)=\adeg_{d-i}(\underline x ^{\underline n};M/W)$$ for all $i=0,\ldots,d-1$,
where $W=H^0_\fkm(M)$. Thus we can assume that $W=0$.

Now we argue by the induction on the dimension $d$ of $M$. In the case $d=2$, $M/D_1$ is a generalized Cohen-Macaulay module by Lemma \ref{lm42}. Therefore $M$ is sequentially generalized Cohen-Macaulay. By Lemma \ref{p1}, $iii)$, there is a $(n_1,n_2)\in \calN(\underline x; M)$ such that $x_1^{n_1},x_2^{n_2}$ is a dd-sequence on $M$.
Therefore, by Lemma \ref{f-dd}, we have 
 $$\ell(H_{\m}^1(M/D_1))=\mid (-1)e_1(x_1^{n_1},x_2^{n_2};M)-\adeg_1(x_1^{n_1},x_2^{n_2};M)\mid\leq C.$$
Hence  $\m^{C}H_{\m}^1(M/D_1)=0$.\\

Fact, suppose that $d>2$ and our assertion holds true for $d-1$. 
Since $\underline n'=(n_1,n_2m_2\ldots,n_dm_d)\in \calN(\underline x; M)$ for all $\underline m=(m_{2},\ldots,m_d) \in \calN(\underline x_1^{\underline n_1}; M/x_1^{n_1}M)$ by Lemma \ref{p1}$(ii)$, 
$\underline x ^{\underline n '}=x_1^{n_1},x_2^{n_2m_2}\ldots,x_d^{n_dm_d}$ is a d-sequence on $M$ because of Lemma \ref{p1}$(i)$. 
Set $y=x_1^{n_1}$ and $\underline y =x_2^{n_2}\ldots,x_d^{n_d}$. Therefore  $y$ is a superficial of $M$ for ideals $(y, \underline y ^{\underline m})$ for all $\underline m \in \calN(\underline y; M/yM)$. 
 It follows by Lemma \ref{lms-e}$(i)$ and Lemma \ref{lmf-a} that
$$
\begin{aligned}
\mid (-1)^ie_i(\underline y^{\underline m} ;M/yM)-\adeg_{d-i-1}(\underline y^{\underline m};M/yM)\mid 
&=\mid (-1)^ie_i(y,\underline y^{\underline m};M)- \adeg_{d-i}(y,\underline y ^{\underline m};M)\mid\\
&=\mid (-1)^ie_i(\underline x ^{\underline n '};M)- \adeg_{d-i}(\underline x ^{\underline n '};M)\mid\leq C,
\end{aligned}
$$
for all  $i=0,\ldots,d-2$ and for all $\underline m \in \calN(\underline y; M/yM)$. 

Now set $\overline M=M/yM$ and and $\overline N=(N+yM)/yM$ for all submodules $N$ of $M$. Let $\calD'=\{D_i'\}_{i=0}^s$ be the dimension filtration of $\overline M$, where $s$ is $t$ if $d_{t-1}>1$ and $t-1$, otherwise.
By the inductive hypothesis, we have $$\fkm^{C}H^j_{\fkm}(\overline M/D^{'}_{i+1})=0$$
for all $\dim D^{'}_{i}=d_i-1\geq 2$ and $j=1,\ldots,d_i-2$. Since  $\calD/{y}\calD\in\calF(M/yM)$ by Lemma \ref{lm41}, we have  $\ell(D_{i+1}^{'}/\overline{D}_{i+1})<\infty$ for all $i=0,\ldots,s-1.$ Therefore we get $H^j_{\fkm}(\overline M/D^{'}_{i+1}) \cong H^j_{\fkm}(\overline M/\overline{D}_{i+1}))$  for all $j\geq 1$
and $i=0,\ldots,s-1$. This implies that
$$\fkm^{C}H^j_{\fkm}(M/(yM+D_{i+1}))=0$$ for all $d_i\geq 3$ and $j=1,\ldots,d_i-2$. 
Since $y$ is $M/D_{i+1}$-regular for all $i = 0, \ldots , t-1$, it follows from short exact sequences 
$$0\longrightarrow \frac{M}{D_{i+1}}\xrightarrow{y^n}\frac{M}{D_{i+1}}\longrightarrow \frac{M}{D_{i+1}+y^nM}\longrightarrow 0$$
that  long sequences
 $$\ldots\longrightarrow H^{j-1}_{\fkm}(\frac{M}{D_{i+1}+y^nM})\longrightarrow H^j_{\fkm}(\frac{M}{D_{i+1}})\xrightarrow{y^n}H^j_{\fkm}(\frac{M}{D_{i+1}})\longrightarrow \ldots$$
 are exact for all $n\ge 1$.
Thus we have $$\m^{C}(0:_{H^j_{\fkm}({M}/{D_{i+1}})}y^n)=0$$ for all $j=2,\ldots,d_i-1$, $d_i\geq 3$ and $n\ge 1$.
Note that $n$ and $C$ are independent of each other. Hence $$\m^{C}H^j_{\fkm}({M}/{D_{i+1}})=0$$
for all $j=2,\ldots,d_i-1$ and $d_i\geq 3$. Moreover, $\ell(H^1_{\m}(M/D_{i+1}))<\infty$ by Lemma \ref{lm42} for all $d_i\geq 2$. So $M$ is a sequentially generalized Cohen-Macaulay module by \cite[Proposition 3.5]{CC2}. 
Therefore there is a $\underline n =(n_1,\ldots,n_d)\in \calN(\underline x; M)$ such that $\underline x ^{\underline n}$ is a dd-sequence on $M$. By Lemma \ref{f-dd},  we havet $\ell(H^1_{\fkm}(M/D_{i+1}))\leq C$ for all $d_i\geq 2$.
 The proof is completed.
\end{proof}
\begin{proof}[Proof of Theorem \ref{thmain}]
$i)\Rightarrow ii)$ is followed by Theorem \ref{thm1}. \\
$ii)\Rightarrow iii)$ is trivial.\\
 $iii)\Rightarrow iv)$ is followed by $\calP_{\calD}(M) \subseteq \calP_{\calF}(M)$ for all $\calF\in \calF(M)$.\\
$iv)\Rightarrow i)$. Let $\underline x= x_1,\ldots,x_d$ be a distinguished system of parameters of $M$ as in Lemma \ref{lm41}.
Since $\calP_{\calD}(M)$ is finite, so is $\{P^{ad}_{\underline x ^{\underline n},M}(n) \mid \underline n \in \calN(\underline x; M)\}$. We have by Theorem \ref{thm2} and \cite[Proposition 3.5]{CC2} that $M$ is  sequentially generalized Cohen-Macaulay, as required.
\end{proof}

\end{document}